\documentclass[10pt,leqno]{amsart}
\usepackage{amsmath}
\usepackage{amsfonts}
\usepackage{amssymb}
\usepackage{graphicx}
\usepackage{color}
\usepackage{hyperref}
\usepackage{xcolor}

\newtheorem{theorem}{Theorem}[section]
\newtheorem*{theorem*}{Theorem}
\newtheorem{lemma}{Lemma}[section]
\newtheorem{corollary}[theorem]{Corollary}
\newtheorem{proposition}{Proposition}[section]

\def\Ha{H_{\alpha}}

\def\a{\alpha}
\def\l{\lambda}

\def\p{\partial}

\def\R{\mathbb{R}}

\def\vol{\operatorname{vol}}

\numberwithin{equation}{section}

\begin{document}

\title[Robin heat kernel]{The Robin heat kernel and its expansion via Robin eigenfunctions}

\author{Yifeng Meng} \thanks{}
\address{School of Mathematical Sciences, Fudan University, Shanghai, 200433, China}
\email{23110180026@m.fudan.edu.cn}

\author{Kui Wang} \thanks{}
\address{School of Mathematical Sciences, Soochow University, Suzhou, 215006, China}
\email{kuiwang@suda.edu.cn}

\subjclass[2020]{35P15, 53C21}

\keywords{Robin heat kernel, Robin eigenvalue, Expansion formula}

\begin{abstract}
 We prove the existence and uniqueness of the Robin heat kernel 
on compact Riemannian manifolds with smooth boundary for Robin parameter $\a\in\R$, expressed as a spectral expansion in terms of Robin eigenvalues and eigenfunctions. 
For the non-negative parameter regime ($\a\ge 0$), we present a direct proof based on trace Sobolev inequalities and eigenfunction estimates. The case of negative parameters ($\a<0$) requires novel analytical techniques to handle $L^\infty$  estimates of Robin eigenfunctions, addressing challenges not present in the non-negative case. Our result extends the  the classical Dirichlet and Neumann cases  to the less-studied negative parameter regime.
\end{abstract}

\maketitle
\section{Introduction}
    Let $M^m$ be an $m$-dimensional compact Riemannian manifold with smooth boundary, and we consider the heat equation 
    \begin{align}\label{1.1}
        u_t(x,t)-\Delta u(x,t)=0, \qquad (x,t) \in M\times (0, T),
    \end{align}
    with the Robin boundary condition
    \begin{align}\label{1.2}
         \frac{\p }{\p \nu} u + \alpha u = 0,  \qquad (x,t) \in \p M\times (0, T), 
    \end{align}
      where $\nu$ denotes the  unit outward normal to $\partial M$, and $\a\in R$ is the Robin parameter. 
   It is well known that for the Neumann boundary ($\a=0$) and the Dirichlet boundary ($\a=+\infty$), the solution of \eqref{1.1} with initial data $u(x,0)=u_0(x)$ can be expressed as
    \begin{equation*}
        u(x,t) = \int_{M} \Ha (x,y,t) u_0(y) dy.
    \end{equation*} 
   where $H_0(x,y,t)$  and $H_{+\infty}(x,y,t)$ represent Neumann and Dirichlet heat kernels, respectively (see \cite{berger1971spectre},\cite[Chapter 10]{li2012geometric}). For further discussion on some space-time boundary conditions, we refer to \cite{jones1972existence}.
  Heat kernels play an important role in the study of partial differential equations and geometric analysis, as evidenced by  \cite{CY81, MC92}  and the comprehensive treatments \cite{Dan00-2, grigor2009heat, Zhang11}. Recent progress on  the heat kernel theory can be found in \cite{ DY23, DY25, LW22, LW17,  zeng2025spectrum}.

    To state our  main result, we first introduce  the Robin eigenvalue problem: let $\l_{i, \a}$ and $\phi_{i,\alpha}(x)$ denote the Robin eigenvalues and eigenfunctions of Laplacian on $M$ with  parameter $\a\in \R$, defined by 
     \begin{align}\label{1.3}
        \begin{cases}
            -\Delta  \phi_{i,\alpha}  = \lambda_{i,\alpha} \phi_{i,\alpha} ,&\quad x \in M ,\\
            \frac{\p }{\p \nu} \phi_{i,\alpha} + \alpha \phi_{i,\alpha} = 0 , & \quad x \in \p M,
        \end{cases}
    \end{align}
where the eigenvalues satisfy
  \begin{equation*}
        \lambda_{1,\alpha} < \lambda_{2,\alpha} \leq \lambda_{3,\alpha} \leq \cdots \to +\infty,
    \end{equation*}
For details, see \cite[Chapter 4]{henrot2017shape}.

Denote by $\Delta_\a$ be the Laplacian with Robin boundary condition \eqref{1.2}. When $\a\ge 0$, the operator $\Delta_\a$  is non-negative, then Theorem 2.1.4 of \cite{Da89} implies the Robin heat kernel has the  following expansion formula
$$
  \Ha(x,y,t) = \sum_{i=1}^{\infty} e^{ - \lambda_{i,\alpha} t } \phi_{i,\alpha}(x) \phi_{i,\alpha}(y).
$$
The expansion formula \eqref{1.4} is well-established for 
Dirichlet and Neumann heat kernels, as well as for heat kernels on closed manifolds, see \cite{berger1971spectre, Da89, li2012geometric}. However, the case $\a<0$ has been less studied. In this paper, we address this remaining case.
    \begin{theorem}\label{thm1}
Let $M$ be a compact Riemannian manifold with smooth boundary, $\a\in \R$,  $\l_{i, \a}$ be the Robin eigenvaules defined by \eqref{1.3},   and $\phi_{i, \a}$ be the corresponding normalized eigenfunctions.  Denote with   
    \begin{align}\label{1.4}
         \Ha(x,y,t) = \sum_{i=1}^{\infty} e^{ - \lambda_{i,\alpha} t } \phi_{i,\alpha}(x) \phi_{i,\alpha}(y).
    \end{align}
Then, $\Ha(x,y,t)$ is well defined on $M \times M \times (0, \infty)$, and is the unique kernel such that: for any  $u_0(x) \in L^2(M)$, the solution of equation \eqref{1.1} with the Robin boundary condition and initial condition $u(x,0)=u_0$ is  given by
    \begin{align*}
        u(x,t) = \int_{M} \Ha (x,y,t) u_0(y) dy.
    \end{align*}
    \end{theorem}

The proof of our main theorem  combines techniques from spectral theory, elliptic regularity, and geometric analysis. For  $\a>0$, we provide a direct proof using a trace Sobolev inequality and iteration methods. The case $\a<0$ presents additional challenges due to the negativity of the principal eigenvalue; we overcome this through a careful decomposition of the first eigenfunction. The uniqueness follows from the  maximum  principle for the Robin heat equation (Theorem \ref{thm:3.1}), which extends classical results for Dirichlet and Neumann boundary conditions.

This paper is organized as follows:  Section \ref{Sect.2} presents preliminary results on Robin eigenvalues, maximum principles, and trace inequalities.
   Section \ref{Sect.3}  contains the proof of our main theorem, with separate treatments for positive and negative Robin parameters.

\section{Preliminaries}\label{Sect.2}
This section establishes the foundational results necessary for our analysis of the Robin heat kernel. We adopt the following notation throughout: for any function $f\in L^q(M)$, its $L^q$ norm is denoted by
    $$
    \Vert f\Vert_q:=(\int_M |f(x)|^q\,dx)^{1/q}.
    $$
\subsection{Robin Eigenvalue Problem.} 
Let $M$ be a compact $m$-dimensional Riemannian manifold with smooth boundary, $\l_{i, \a}$ ($i=1,2, \cdots$) be the Robin eigenvalues, and  $\phi_{i, \a}$ be  normalized eigenfunctions such that $\Vert \phi_{i,\alpha}\Vert_2 =1$.
It is well known that the eigenvalue problem \eqref{1.3} is equivalent to the following variation problem 
    \begin{equation}{\label{2.1}}
        \lambda_{i,\alpha}(M) = \inf_{\substack{H \subset H^1(M),\\ \dim H=i} } \sup_{0 \neq u \in H} \frac{\int_{M} |\nabla u|^2 dx + \alpha \int_{\partial M} u^2 dS }{ \int_{M} u^2 dx  },
    \end{equation}
    and particularly
    \begin{equation}{\label{2.2}}
        \lambda_{1,\alpha}(M) = \inf_{0 \neq u \in H^1(M)}  \frac{\int_{M} |\nabla u|^2 \, dx + \alpha\int_{\partial M}  u^2 \, dS }{ \int_{M} u^2 \, dx},
    \end{equation}
    where $dS$ is the induced measure on $\p M$.
   Moreover, the eigenfunctions $\phi_{i,\alpha}$($i=1,2,\cdots$) form a complete  orthonormal basis for  $L^2(M)$. If $M$ is connected, the Krein-Rutman Theorem guarantees the simplicity of $\l_{1,\a}$ and strict positivity of its eigenfunction (see \cite[Section 4.2]{henrot2017shape}). When $\a<0$, we obtain the following  additional information.
     \begin{proposition}{\label{prop:posi}}
Suppose $\a<0$, and denote by $\phi_{1,\alpha}$ be the  positive normalized eigenfunction  with respect to $\lambda_{1,\alpha}$. Then, $\lambda_{1,\alpha}(M)<0$ and $\inf_{M} \phi_{1,\alpha} >0$.
     \end{proposition}
    \begin{proof}
        Taking $u=1$ as a trial 
        function in \eqref{2.2} yields
        $$
        \l_{1,\a}(M)\le \frac{\alpha\int_{\partial M}   \, dS }{ \int_{M}  \, dx}<0.
        $$
     Suppose $\inf_M \phi_{1,\alpha}=0$, then there exists a $x_0 \in \p M$ such that $\phi_{1, \a}(x_0)=0$. The Robin boundary condition \eqref{1.2} gives $\p_{\nu} \phi_{1,\alpha}(x_0) = 0$, which contradicts with the Hopf lemma. Hence $\inf_M \phi_{1,\alpha}>0$.
    \end{proof}

\subsection{Maximum Principles and Uniqueness}\label{sect 2.2}
The Robin heat equation satisfies the following maximum principle:
    \begin{theorem}{\label{thm:3.1}}
        Let $M$ be a compact manifold with smooth boundary. Suppose that $u(x,t)\in C^{2,1}(M\times [0,\infty))$ satisfies 
    \begin{align}\label{2.3}
        \begin{cases}
            \left( \p_t  -\Delta \right) u \ge 0 ,&\quad (x,t) \in M\times (0,\infty),\\
            \frac{\p }{\p \nu} u + \alpha u \ge 0 , & \quad (x, t) \in \p M\times (0,+\infty),\\
            u(x,0) \geq 0, &\quad x\in M. 
        \end{cases}
    \end{align}
    Then, $$u(x,t) \geq 0$$ for all $x\in M$ and $t>0$.
    \end{theorem}
    \begin{proof}
   Let $\beta=\min\{\alpha , 0\}-1$, and denote by $\lambda_{1,\beta}$ and $\phi_{1,\beta}$  the first Robin eigenvalue and its corresponding positive normalized eigenfunction. Clearly
       $$
       w(x,t):=e^{-\l_{1, \beta } t}\phi_{1,\beta}(x)
       $$
is strictly positive and satisfies
    \begin{align}\label{2.4}
        \begin{cases}
            \left( \p_t  -\Delta \right) w =  0 ,&\quad (x,t) \in M\times (0,\infty),\\
            \frac{\p }{\p \nu} w + \beta w= 0 , & \quad x \in \p M. \\
        \end{cases}
    \end{align}
For any $\varepsilon >0$, define $v_\varepsilon(t,x):= u(t,x)+\varepsilon w(t,x)$. By \eqref{2.3} and \eqref{2.4}, we derive that
    \begin{align*}
        \begin{cases}
            \left( \p_t  -\Delta \right) v_\varepsilon   \ge  0 ,&\quad (x,t) \in M\times (0,\infty),\\
            \frac{\p }{\p \nu} v_\varepsilon  + \alpha v_\varepsilon  \ge \varepsilon (\alpha - \beta) w > 0 , & \quad x \in \p M, \\
            v_\varepsilon  > 0, & \quad t = 0.
        \end{cases}
    \end{align*}
    We claim that $v_\varepsilon(x,t)>0$ for all $((x,t)\in [0,\infty)\times M$. If not, let $t_0$ be the first time such that  $v_\varepsilon (x_0,t_0)=0$ for some $x_0\in M \cup \p M$, and by strong parabolic maximal principle, we have $x_0\in \p M$. Then we have 
    \begin{align*}
        0 \geq \p_{\nu} v_\varepsilon (x_0,t_0) 
         = -\alpha v_\varepsilon (x_0,t_0) + \varepsilon (\alpha - \beta) w(x_0,t_0)=\varepsilon (\alpha - \beta) w(x_0,t_0), 
    \end{align*}
    contradicting with $ w(x_0,t_0)>0$. Hence
    $$
    u(x,t)+\varepsilon w(x,t)>0,\quad (x, t) \in M\times (0,\infty) 
    $$
   for all $\varepsilon>0$. Then  letting $\varepsilon\to 0$, we have
   $$
   u(x,t)\ge 0,
   $$
proving the theorem.
    \end{proof}
 The following corollary comes true directly from Theorem \ref{thm:3.1}. 
    \begin{corollary}[Uniqueness and Positivity]{\label{Cor:uni}}
       The solution to the Robin heat equation with given initial data is  unique, and the kernel function $\Ha(x,y,t) \geq 0$ almost everywhere if exists.
    \end{corollary}
\subsection{Trace  Sobolev Inequalities} \label{sect 2.3}
The analysis of Robin boundary conditions requires careful control of boundary terms. We establish the following fundamental inequalities:
    \begin{lemma}{\label{lm:1}}
         Let $M$ be a compact manifold with smooth boundary. Then there exists a constant $C_{1}$, depending on $M$, such that for all $u\in H^1(M)$ it holds
         \begin{align}\label{2.6}
             \int_{\partial M} u^2 dS \leq C_{1} \left( \Vert \nabla u\Vert _2 \Vert u\Vert_2 + \Vert u\Vert _2^2 \right).
         \end{align}
    \end{lemma}
    \begin{proof}
       Recall from  standard trace theorem \cite[Chapter 5] {evans2022partial} that for all $v\in W^{1,1}(M)$ it holds
        \begin{equation*}
            \int_{\partial M} |v| dS \leq C(M) \left( \int_M |\nabla v| + \int_{M} |v| dx \right),
        \end{equation*}
       and taking $v= u^2$ gives \eqref{2.6}.
    \end{proof}
    \begin{lemma}{\label{thm:high} }
        There exist  positive constants $C_2=C_2(M)$(depending on $M$) and $C_3=C_3(M,\a)$ (depending on $\alpha$ and $M$),  such that
        \begin{equation}\label{2.7}
            \lambda_{k,\alpha}(M) \geq C_2 k^{ \frac{1}{m-1} } - C_{3}.
        \end{equation}
    \end{lemma}
    \begin{proof}
Recall from Theorem 10.1 of \cite{li2012geometric} that there exists a positive constant $C$, depending on the volume of $M$,  such that
    \begin{equation}\label{2.8}
            \lambda_{k,0}(M) \geq C k^{\frac{1}{m-1}}.
        \end{equation}
  
    If $\alpha \geq 0$, it follows from \eqref{2.1} that $\l_{k, \a}$ is monotone increasing in $\a$, hence using   \eqref{2.8}  we have
        \begin{align}\label{2.9}
            \lambda_{k,\alpha}(M) \geq \lambda_{k,0}(M)\ge C k^{\frac{1}{m-1}}. 
        \end{align}
 
  If $\alpha < 0$, from Lemma \ref{lm:1} we know that there exists a constant $C_1=C_1(M)$ such that
        \begin{align*}
            \int_{\partial M} u^2 dS &\leq C_1(\Vert \nabla u\Vert_2 \Vert u\Vert_2+\Vert  u\Vert_2^2)\\
            & \le  -\frac{1}{2 \alpha } \Vert\nabla u\Vert_2^2-\frac{\a}{2}C_1^2\Vert  u\Vert_2^2 +C_1||u||_2^2\\
            &:= -\frac{1}{2 \alpha } \Vert\nabla u\Vert_2^2+\frac{C_3}{\a}||u||_2^2
        \end{align*}
      where $C_3$ is a positive constant depending on $\alpha$ and $M$. Plugging above inequality into \eqref{2.1} we have
        \begin{align}\label{2.10}
        \begin{split}
        \lambda_{k,\alpha}(M) &= \inf_{\substack{H \subset H^1(M),\\ \dim H=k} } \sup_{0 \neq u \in H}\frac{\int_{M} |\nabla u|^2 dx + \int_{\partial M} \alpha u^2 dS }{ \int_{M} u^2 dx  }  \\ 
        & \geq \inf_{\substack{H \subset H^1(M),\\ \dim H=k} } \sup_{0 \neq u \in H}  \frac{1}{2} \frac{\int_{M} |\nabla u|^2 dx }{  \int_{M} u^2 \, dx }-C_3 \\
        & = \frac{1}{2} \lambda_{k,0}(M)-C_3\\
        &\ge \frac 1 2 C k^{\frac 1 {m-1}}-C_3
        \end{split}
        \end{align}
    for $\a<0$,    where we used \eqref{2.8} in the last inequality.
From \eqref{2.9} and \eqref{2.10}, we conclude 
\eqref{2.7} holds by choosing $C_2:=\frac 1 2 C$.
    \end{proof}

In this subsection, we use compactness argument to prove a trace Sobolev inequality, which will be used to handle the Robin boundary condition in the proof of Theorem \ref{thm1}.  
\begin{lemma}{\label{lm2.3}}
    Let $M$ be a compact $m$-dimensional manifold with  smooth boundary. Then there exists a  positive constant $C_{4}>0$, depending on $M$, such that for any $f \in H^1(M)$ it holds
  \begin{equation}\label{2.11}
        \int_{M} |\nabla f|^2 \, dx + \int_{\p M} |f|^2 dS \geq C_{4} 
         \left( \int_{M} |f|^{\frac{2m}{m-2}} \right)^{\frac{m-2 }{m}}
    \end{equation}
if $m\ge 3$;  and if $m=2$,
 \begin{equation}\label{2.12}
        \int_{M} |\nabla f|^2 dx + \int_{\p M} |f|^2 dS \geq C_{4} \left( \int_{M} |f|^p \right)^{\frac{1}{p}}
    \end{equation}
    for any given $p>2$ with constant $C_4$ depending on $p$.
\end{lemma}
\begin{proof}
We  prove  \eqref{2.11} via compactness argument.    Suppose  \eqref{2.11}  fails, then we can choose a sequence $\{f_{k}\}_{k=1}^{\infty} \subset H^1(M)$ satisfying 
 $\Vert f_k\Vert_\frac{2m}{m-2}=1$
and
 \begin{align}\label{2.13}
        \Vert f_k\Vert _{L^2(\p M)} + \Vert \nabla f_k\Vert_2 \leq \frac{1}{k}.
    \end{align}

 On one hand,  Using H\"older inequality, we estimate that
    \begin{align}\label{2.14}
 \Vert f_k\Vert_2 \le  \Vert f_k\Vert_{\frac{2m}{m-2}} \cdot\vol(M)^{\frac{1}{m}}
        = \vol(M)^{\frac{1}{m}}. 
    \end{align}
According to \eqref{2.13} and \eqref{2.14}, we see that $f_k$ is uniformly bounded in $H^1(M)$, therefore there exists  a subsequence $f_{k_s}$ converges to $f_0$ in $H^1(M)$ as $s\to \infty$. Moreover, by \eqref{2.13}, we have $f_0(x)=0$.  
  
  On the other hand, by Sobolev embedding theorem,  we have
  $f_{k_s}$ converges to $f_0$ in $L^{\frac{2m}{m-2}}(M)$ as $s\to \infty$, so
    \begin{align*}
        \Vert f_{k_s}\Vert_\frac{2m}{m-2} \to 0,
    \end{align*}
  contradicting with $\Vert f_{k_s}\Vert_\frac{2m}{m-2}=1$. Hence \eqref{2.11} comes true. 
  
If $m = 2$, \eqref{2.12} holds by the similar argument, we omit the details.
\end{proof}

\section{Proof of Theorem \ref{thm1}}{\label{Sect.3}}
This section presents the detailed proof of our main result, establishing the existence and uniqueness of the Robin heat kernel for all $\a\in \R$. While the case of positive $\a$ is essentially covered by Theorem 2.1.4 of \cite{Da89}, we provide a complete and self-contained treatment for both positive and negative parameters to ensure full mathematical rigor and to highlight the distinct technical challenges that emerge in each regime.

\subsection{The Case of Positive Robin Parameter}
We begin by establishing uniform estimates for Robin eigenfunctions, which are crucial for controlling the convergence of the heat kernel expansion. See also \cite{Dan00-1, vB14} for the previous results on Robin boundary problems with positive Robin parameters, where energy methods were used to prove the existence results.
\begin{lemma}\label{lm3.1}
     Let $M$ be a compact $m$-dimensional manifold with smooth boundary. For $m\ge 3$, define $\gamma=\frac{m}{m-2}$; for $m=2$, let  $\gamma>2$ be  arbitrary. Let $\l_{i, \a}$ be the $i$th Robin eigenvalue with Robin parameter $\a>0$, $\phi_{i, \a}$ be the corresponding positive normalized eigenfunction, and $C_4$ be the constant defined in Lemma \ref{lm2.3}.
       Then, the $L^\infty$ norm of $\phi_{i, \a}$ satisfies
       \begin{align} \label{3.1}
         \Vert \phi_{i, \a}\Vert_{\infty}\le C_5 \l_{i, \a}^{\frac{1}{2} \frac{\gamma}{\gamma-1}},
     \end{align}
     where $C_5= \gamma ^{ \frac{1}{2} \frac{\gamma}{(\gamma-1)^2} } \left(\frac{2}{C_{4} \min\{1,\alpha\}}\right) ^{\frac{1}{2} \frac{\gamma}{\gamma-1}}$.
\end{lemma}
\begin{proof}
Let $f=|\phi_{i,\alpha}|$. The eigenvalue equation $-\Delta \phi_{i,\alpha}=\l_{i, \a} \phi_{i,\alpha}$ implies
    \begin{equation*}
        \Delta f \geq - \lambda_{i,\alpha} f
    \end{equation*}
    in the distribution sense, hence for all $k\ge 2$ it holds
    \begin{align}\label{3.2}
        -\int_M f^{k-1}\Delta f\le \l_{i, \a}\int_M f^k.
    \end{align}
Using integration by parts, we estimate that 
    \begin{align}\label{3.3}
    \begin{split}
            \int_{M} f^{k-1} \Delta f &= -(k-1) \int_{M} f^{k-2} |\nabla f|^2 + \int_{\p M} f^{k-1} \frac{\p f}{\p \nu} dS \\
            &= -\frac{4(k-1)}{k^2} \int_{M} |\nabla f^{\frac{k}{2}  } |^2 - \alpha \int_{\p M} |f^{\frac{k}{2}}|^2 dS\\
            & \leq -\frac{2\min\{1,\alpha\}}{k} \left( \int_{M} |\nabla f^{\frac{k}{2}  }   |^2  + \int_{\p M} |f^{ \frac{k}{2} }|^2 dS \right),
        \end{split}
    \end{align}
    where we used $k\ge 2$ in the last inequality. 
  Recall from  Lemma \ref{lm2.3} that
   \begin{align*}
    \int_{M} |\nabla f^{\frac{k}{2}  }   |^2  + \int_{\p M} |f^{ \frac{k}{2} }|^2 dS\ge C_{4} \left(\int_{ M} |f|^{ k\gamma }\right)^{1/\gamma},
\end{align*}
then we conclude from  \eqref{3.3} and  the above inequality that
     \begin{equation}\label{3.4}
            \int_{M} f^{k-1} \Delta f 
             \leq -\frac{C_{4} \cdot \min\{1,\alpha\}}{k} \left( \int_{M} |f|^{ k\gamma } \right)^{1/\gamma }.
    \end{equation}
Putting \eqref{3.2} and \eqref{3.4} together, we obtain 
    \begin{equation*}
        \int_{M} |\phi_{i,\alpha}|^k \geq \frac{C_{4} \cdot \min\{1,\alpha\}}{k \lambda_{i,\alpha}} \left( \int_{M} |\phi_{i,\alpha}(x)|^{k\gamma } \right)^{1/\gamma},
    \end{equation*}
i.e.
    \begin{equation}\label{3.5}
        \Vert \phi_{i,\alpha} \Vert_{k\gamma} \leq \left( \frac{k \lambda_{i,\alpha} }{C_{4} \cdot \min\{1,\alpha\}} \right)^{\frac{1}{k} }  \Vert\phi_{i,\alpha} \Vert_k,
    \end{equation}
    and substituting  $k = 2 \gamma^j$ for $j=0,1,2,\cdots$ in \eqref{3.6}, we get
    \begin{align*}
        \Vert \phi_{i,\alpha}\Vert_{2\gamma^{j+1}} \leq \left( \frac{2\gamma^j \lambda_{i,\alpha} }{C_{4} \cdot \min\{1,\alpha\} } \right)^{ \frac{1}{2\gamma^j}}  \Vert \phi_{i,\alpha} \Vert _{2\gamma^j}.
    \end{align*}
Observing  that $\Vert\phi_{i,\alpha}\Vert_2 =1$, we have
    \begin{align*}
        \Vert\phi_{i,\alpha}\Vert_{2\gamma^j} \leq \prod_{l=0}^j  \left( \frac{2\gamma^l \lambda_{i,\alpha} }{C_{4} \cdot \min\{1,\alpha\} } \right)^{ \frac{1}{2\gamma^l}},
    \end{align*}
   and let $j\to \infty$ in above inequality, we derive 
    \begin{align*}
        ||\phi_{i,\alpha}||_{\infty} \leq \gamma ^{ \frac{1}{2} \frac{\gamma}{(\gamma-1)^2} } \left(\frac{2\lambda_{i,\alpha}}{C_{4} \cdot \min\{1,\alpha\}}\right) ^{\frac{1}{2} \frac{\gamma}{\gamma-1}}
        = C_5 \lambda_{i,\a} ^{\frac{1}{2} \frac{\gamma}{\gamma-1}},
    \end{align*}
  proving \eqref{3.1}.
\end{proof}
\begin{proof}[Proof of Theorem \ref{thm1} for $\a>0$]
    We only prove the case for $m\ge 3$, and for $m=2$ the argument is similar. In which case, we have
    \begin{align*}
        \Vert \phi_{i, \a}\Vert_{\infty}\le C_5 \l_{i, \a}^{m/4},
    \end{align*}
    by Lemma \ref{lm3.1}.
    Let
    $$
    d(t):=\sqrt{\frac{m^m}{e^m}}\frac 1 {t^{m/2}}, \quad t>0,
    $$
then it follows easily that 
    \begin{align}\label{3.6}
        e^{-xt} x^{\frac{m}{2} } \leq d(t) e^{-\frac{xt}{2} }, \quad x>0, \quad t>0.
    \end{align}
Using \eqref{3.1} and \eqref{3.6},  we estimate that 
    \begin{align*}
   |e^{- \lambda_{i,\alpha} t } \phi_{i,\alpha}(x) \phi_{i,\alpha}(y)| \leq& e^{- \lambda_{i,\alpha} t } ||\phi_{i,\alpha}||_{\infty}^2 
         \leq C_5^2 e^{- \lambda_{i,\alpha} t} \lambda_{i,\alpha}^{\frac{m}{2}}\\
        \leq& C_5^2 d(t) e^{-\frac{\lambda_{i,\alpha} t }{2} } 
         \leq C_5^2 d(t) e^{- \frac{C_2 i^{\frac{1}{m-1} }t }{ 2}  }  
    \end{align*}
    where we used \eqref{2.9} in the last inequality. Hence we have
    \begin{align*}
         \Ha(x,y,t): = \sum_{i=1}^{\infty} e^{ - \lambda_{i,\alpha} t } \phi_{i,\alpha}(x) \phi_{i,\alpha}(y)
    \end{align*}
    converges uniformly in $M \times M \times [\varepsilon, \infty)$ for any $\varepsilon>0$. Since
    \begin{align*}
        \int_{M} \langle \nabla \phi_{i,\alpha} , \nabla \phi_{j,\alpha} \rangle + \alpha \int_{\p M} \phi_{i,\alpha} \phi_{j,\alpha} = \delta_{ij} \lambda_{i,\alpha},
    \end{align*}
    then 
    \begin{align*}
        &\int_{M} |\sum_{i=1}^k e^{- \lambda_{i,\alpha} t} \phi_{i,\alpha}(x) \nabla \phi_{i,\alpha}(y) |^2 + \alpha \int_{\p M} |\sum_{i=1}^k e^{- \lambda_{i,\alpha} t} \phi_{i,\alpha}(x) \phi_{i,\alpha}(y) |^2 \\
         = &\sum_{i=1}^k e^{-2 \lambda_{i,\alpha} t } \lambda_{i,\alpha} \phi_{i,\alpha}(x) \phi_{i,\alpha}(x), 
    \end{align*}
    which is uniformly bounded for any $k>0$. Since the truncated sums 
    \begin{equation*}
    \sum_{i=1}^k e^{- \lambda_{i,\alpha} t} \phi_{i,\alpha}(x) \phi_{i,\alpha}(y)
    \end{equation*}
    satisfy the heat equation and the  Robin boundary condition, the limit function $\Ha(x,y,t)$  inherits these properties as a weak solution, which by regularity theory becomes smooth. Moreover, for any given $u_0(x)\in L^2(M)$, $u(x,t):=\int_M H_\a(x,y, t)u_0(y)\, dy$ is a solution of \eqref{1.1} with the Robin boundary condition and $\lim_{t\to 0^+} u(x,t)=u_0(x)$.
    
    Theorem \ref{thm:3.1} asserts that $f(x,t)$ is positive on $(M \setminus \p M) \times (M \setminus \p M) \times (0, \infty)$ whenever $f_0 \geq 0$ on $M$. In addition, the Robin boundary condition and Corollary \ref{Cor:uni} give the uniqueness of the heat kernel, since there is only one solution with given initial data. Hence, we complete the proof of Theorem \ref{thm1} for $\a>0$. 
\end{proof} 
\subsection{The Case of Negative Robin Parameter}
When Robin parameter $\a<0$, the proof of Lemma \ref{lm3.1} is invalid since the trace Sobolev inequality cannot be directly applied. Hence, the argument for $\a>0$ does not apply to the case $\a<0$. Fortunately,  we consider the eigenvalue gap $\l_{i, \a}-\l_{1,\a}$ to overcome technical difficulties. To begin with, we recall the following well-known Sobolev inequality. 
\begin{lemma}\label{lm3.2}
    Let $M$ be a  complete $m$-dimensional manifold,  possibly with boundary, and $\gamma$ be the constant defined in Lemma \ref{lm3.1}.\\
(1)
For $m \ge 3$, there exists a positive constant
$C_6$ depending  on the Neumann $\frac{m}{m-1}$-Sobolev constant of $M$ (see \cite[Definition 9.4]{li2012geometric}), and $C_7$ depending on the volume of $M$,  such that
\begin{align}\label{3.7}
    \int_M |\nabla f|^2\ge C_6\Big(\big(\int_M |f|^{2\gamma}\big)^{\frac{1}{\gamma}}-C_7\int_M f^2\Big)
\end{align}
for all $f\in H^{1,2}(M)$. \\
(2) For $m=2$, \eqref{3.7} holds with  positive constants
$C_6$ and $C_7$, depending on $\gamma$.
\end{lemma}
\begin{proof}
See Corollary 9.3 in \cite{li2012geometric}.
\end{proof}

\begin{lemma}\label{lm3.3}
     Let $M$ be a compact $m$-dimensional manifold with smooth boundary,  $\l_{i, \a}$ be the $i$th Robin eigenvalue with $\a<0$, and $\phi_{i, \a}$ be the corresponding positive normalized eigenfunction. Let $\gamma$,
     $C_{6}$ and $C_7$ be the  constants from Lemma \ref{lm3.2},  and $C_8=\max_{x\in M}\{|\frac{\nabla \phi_{1,\a}(x)}{\phi_{1, \a}(x)}|\}$  (positive  by Proposition \ref{prop:posi}). Then the $L^\infty$ norm of $\phi_{i, \a}$ satisfies
     \begin{align}\label{3.8}
         ||\phi_{i,\alpha}||_{\infty} \leq \frac{\exp\{ \frac{C_6 C_7}{4C_8^2} \cdot \frac{\gamma^2}{\gamma^2-1} + \frac{1}{2} \frac{\gamma}{(\gamma-1)^2} \log \gamma \}}{(C_6/2)^{\frac{1}{2} \frac{\gamma}{\gamma-1}}}\cdot \frac{\sup_M \phi_{1,\alpha} }{\inf_M \phi_{1,\alpha} } (\lambda_{i,\alpha} - \lambda_{1,\alpha} + 4C_{8}^2)^{\frac{1}{2} \frac{\gamma}{\gamma-1}}.
     \end{align}
\end{lemma}
\begin{proof}
    Denote by $\phi_{i,\alpha}(x)$ the normalized eigenfunctions with Robin eigenvalue $\l_{i, \a}$, and 
let
    \begin{align*}
        w_i(x) = \frac{\phi_{i,\alpha}(x) }{\phi_{1,\alpha}(x)}.
    \end{align*}
    It can be easily checked that
    \begin{align}{\label{3.9}}
        \begin{cases}
            \Delta w_i(x) + 2 \langle \nabla \log \phi_{1,\alpha}(x), \nabla w_i(x) \rangle + (\lambda_{i,\alpha} - \lambda_{1,\alpha}) w_i(x) = 0, \quad x \in M, \\
            \p_\nu w_i(x) = 0 , \quad x \in \p M.
        \end{cases}
    \end{align}
     Let $u(x) = |w_i(x)|$, and using \eqref{3.9}  we estimate that 
 \begin{align*}
            \Delta u (x) &= \Delta |w_i(x)|  \geq - |\Delta w_i(x)| \\
            & = - |2 \langle \nabla \log \phi_{1,\alpha}(x), \nabla w_i(x) \rangle + (\lambda_{i,\alpha} - \lambda_{1,\alpha}) w_i(x)| \\
            & \geq -2 |\nabla \log \phi_{1,\alpha}(x)| |\nabla u| - (\lambda_{i,\alpha} - \lambda_{1,\alpha}) u\\
            &\geq -2 C_{8} |\nabla u(x) | - (\lambda_{i,\alpha} - \lambda_{1,\alpha}) u(x),
\end{align*}
    where we used 
     Kato's inequality  in the first inequality. 
 For $k\ge 2$,   multiplying $u(x)^k$ and integrating  over $M$ yields
    \begin{align*}
        \int_M u^{k-1} \Delta u dx \geq -2 C_{8} \int_M u^{k-1}|\nabla u|dx  - (\lambda_{i,\alpha}-\lambda_{1,\alpha}) \int_M u^k dx,
    \end{align*}
and integration by parts gives 
    \begin{align*}
        (k-1) \int_M u^{k-2}|\nabla u|^2 dx \leq 2C_{8} \int_M u(x)^{k-1}|\nabla u(x)| dx + (\lambda_{i,\alpha}-\lambda_{1,\alpha}) \int_M u^k dx.
    \end{align*}
    Observing that
    \begin{align*}
        2 u^{k-1} |\nabla u| \leq \frac{1}{2C_{8}} u^{k-2} |\nabla u|^2 + 2C_{8} |u|^k,
    \end{align*}
    we have
    \begin{equation}{\label{3.10}}
        (k-\frac{3}{2}) \int_M u^{k-2}|\nabla u|^2 dx \leq (\lambda_{i,\alpha} - \lambda_{1,\alpha} + 4C_{8}^2) \int_{M} u^k dx.
    \end{equation}
Using Sobolev inequality \eqref{3.7}, we obtain 
    \begin{align*}
        \int_M u^{k-2} |\nabla u|^2 dx = \frac{4}{k^2}        
        \int_M |\nabla (u^{k/2})|^2 dx 
        \geq \frac{4C_6}{k^2}  \Big((\int_M |u|^{k\gamma} dx)^{1/\gamma} - C_7 \int_M |u|^k dx \Big),
    \end{align*}
    where $C_6$ is defined Lemma \ref{lm3.2}.
Plugging  above inequality into \eqref{3.10}, we have
    \begin{align*}
        \left( \int_M |u|^{k\gamma} \right)^{1/\gamma} &\leq \Big(\frac{k^2}{4(k-3/2) C_{6}}(\lambda_{i,\alpha} -\lambda_{1,\alpha} + 4C_{8}^2) + C_7\Big ) \int_M |u|^k dx \\
        & \leq \left(C_7+ \frac{1}{C_{6}}(\lambda_{i,\alpha} - \lambda_{1,\alpha} + 4C_{8}^2)k \right) \int_M |u|^k dx,
    \end{align*}
    where we used $k\ge 2$ in the last inequality.
     Hence we conclude
    \begin{align}\label{3.11}
        ||u||_{\gamma k} \leq  \left(C_7+ \frac{1}{C_{6}}(\lambda_{i,\alpha} - \lambda_{1,\alpha} + 4C_{8}^2)k \right)^{1/k} ||u||_k
    \end{align}
    for $k\ge 2$. Let $a= \frac{\lambda_{i,\alpha} - \lambda_{1,\alpha} + C_{8}^2}{C_{6}}$, \eqref{3.11} becomes to
    \begin{align}\label{3.12}
        ||u||_{\gamma k} \leq (C_7 + ak)^{1/k} ||u||_k.
    \end{align}
  Choosing  $k = 2 \gamma ^{j-1}$ for $j=1,2,\cdots$, we obtain from \eqref{3.12} that
    \begin{align*}
        ||u||_{2\gamma^j} = (C_7 + 2a \gamma^{j-1} )^{\frac{1}{2\gamma^{j-1}}} ||u||_{2\gamma^{j-1}},
    \end{align*}
    which implies that
    \begin{align}\label{3.13}
    \begin{split}
        ||u||_{\infty} \leq & \prod_{j=1}^{+\infty} (C_7 + 2a\gamma^{j-1} )^{\frac{1}{2\gamma^{j-1}}} ||u||_2 \\
        \leq & \exp\{ \frac{C_6C_7}{4C_8^2} \cdot \frac{\gamma^2}{\gamma^2-1} + \frac{1}{2} \frac{\gamma}{(\gamma-1)^2} \log \gamma \} (2a)^{\frac{1}{2} \frac{\gamma}{\gamma-1}} / \inf_M \phi_{1,\alpha},\\
        =:&  \frac{C_9}{\inf_M \phi_{1,\alpha} } (\lambda_{i,\alpha} - \lambda_{1,\alpha} + 4C_{8}^2)^{\frac{1}{2} \frac{\gamma}{\gamma-1}},
    \end{split}
    \end{align}
    where 
   \begin{align}\label{3.14}
   C_9:= \frac{\exp\{ \frac{C_6C_7}{4C_8^2} \cdot \frac{\gamma^2}{\gamma^2-1} + \frac{1}{2} \frac{\gamma}{(\gamma-1)^2} \log \gamma \}}{(C_6/2)^{\frac{1}{2} \frac{\gamma}{\gamma-1}}}
\end{align}
  Therefore using the definition of $u$ and \eqref{3.13} we conclude that
    \begin{align*}
   ||\phi_{i,\alpha}||_{\infty} \leq ||u||_{\infty} ||\phi_{1,\alpha}||_{\infty} \leq C_9 \cdot \frac{\sup_M \phi_{1,\alpha} }{\inf_M \phi_{1,\alpha} } (\lambda_{i,\alpha} - \lambda_{1,\alpha} + 4C_{8}^2)^{\frac{1}{2} \frac{\gamma}{\gamma-1}},
    \end{align*}
proving \eqref{3.8}.
\end{proof}

\begin{proof}[Proof of the case for $\a<0$]
We only prove the case for $m\ge 3$. In which case, we have
    \begin{align*}
          ||\phi_{i,\alpha}||_{\infty} \leq C_9 \cdot \frac{\sup_M \phi_{1,\alpha} }{\inf_M \phi_{1,\alpha} } (\lambda_{i,\alpha} - \lambda_{1,\alpha} + 4C_{8}^2)^{m/4},,
    \end{align*}
    by Lemma \ref{lm3.1}.
Let  $$h(t):=\sqrt{\frac{m^m}{e^m}}\frac{e^{2C_8^2t}}{t^{m/2}},$$
where $C_8$ is the constant defined as in Lemma \ref{lm3.3}. Then direct calculation gives
    \begin{equation}\label{3.15}
        e^{-xt} (x +4C_{8}^2)^{\frac{m}{2} } \leq h(t) e^{-\frac{xt}{2} }
    \end{equation}
for $x>0$ and $t>0$. Using \eqref{3.8} and \eqref{3.15} we estimate that     
\begin{align*}
    \begin{split}
        |e^{- \lambda_{i,\alpha} t } \phi_{i,\alpha}(x) \phi_{i,\alpha}(y)| &\leq e^{- \lambda_{i,\alpha} t } ||\phi_{i,\alpha}||^2_{\infty} \\
        & \leq (C_9\frac{\sup_M \phi_{1,\alpha} }{\inf_M \phi_{1,\alpha} })^2e^{- \lambda_{1,\alpha} t} e^{- (\lambda_{i,\alpha} - \lambda_{1,\alpha}) t} (\lambda_{i,\alpha} - \lambda_{1,\alpha} + 4C_{8})^{\frac{m}{2}}\\
        & \leq (C_9\frac{\sup_M \phi_{1,\alpha} }{\inf_M \phi_{1,\alpha} })^2h(t) e^{-\lambda_{1,\alpha}t} e^{-\frac{\lambda_{i,\alpha} - \lambda_{1,\alpha}}{2}t} \\
        & \leq (C_9\frac{\sup_M \phi_{1,\alpha} }{\inf_M \phi_{1,\alpha} })^2h(t) e^{-\frac{\lambda_{1,\alpha} }2t} e^{-C_2 i^{\frac{1}{m-1}} t+C_3 t/2},
    \end{split}
    \end{align*}
    where we used \eqref{2.7} in the last inequality,  $C_2$ and $C_3$ are positive constants defined in Lemma \ref{thm:high},  $C_6$ is the constant defined in Lemma \ref{lm3.3}, and $C_9$ is defined in \eqref{3.14}. Hence
    \begin{align*}
         \Ha(x,y,t) = \sum_{i=1}^{\infty} e^{ - \lambda_{i,\alpha} t } \phi_{i,\alpha}(x) \phi_{i,\alpha}(y)
    \end{align*}
    converges uniformly in $M \times M \times [\varepsilon, \infty)$ for any $\varepsilon>0$. 
    Observing that 
    \begin{align*}
        \int_{M} \langle \nabla \phi_{i,\alpha} , \nabla \phi_{j,\alpha} \rangle + \alpha \int_{\p M} \phi_{i,\alpha} \phi_{j,\alpha} = \delta_{ij} \lambda_{i,\alpha},
    \end{align*}
    we get
    \begin{equation}{\label{3.16}}
    \begin{split}
        &\int_{M} |\sum_{i=1}^k e^{- \lambda_{i,\alpha} t} \phi_{i,\alpha}(x) \nabla \phi_{i,\alpha}(y) |^2 + \alpha \int_{\p M} |\sum_{i=1}^k e^{- \lambda_{i,\alpha} t} \phi_{i,\alpha}(x) \phi_{i,\alpha}(y) |^2 \\
         = &\sum_{i=1}^k e^{-2 \lambda_{i,\alpha} t } \lambda_{i,\alpha} \phi_{i,\alpha}(x) \phi_{i,\alpha}(x),
        \end{split}
    \end{equation}
    which is uniformly bounded for any $k>0$.
    
   The remainder of the proof mirrors the 
$\a>0$ case, with the truncated sums satisfying \eqref{3.16} and their limit inheriting the solution properties.
\end{proof}

\section*{Acknowledgments}
The first  author would like to express his sincere gratitude to his supervisor, Professor Bobo Hua, for lots of encouragement and helpful suggestions. 
The authors would like to thank Professor Hongjie Dong  for helpful discussion on the maximum principle for the Robin heat equation in Section \ref{sect 2.2}, and
Professor Genggeng Huang for helpful discussion regarding the trace Sobolev inequalities in Section \ref{sect 2.3}, and Professor Michiel  van den Berg for bringing \cite{vB14} and other relevant works on Robin boundary problems to our attention. The authors are also thankful to Professors Xianzhe Dai, Xiaolong Li, and Qi S. Zhang for their interest in this work and their valuable feedback.
The research of this paper is partially supported by 
 NSF of Jiangsu Province Grant No. BK20231309.
 
\bibliographystyle{plain}
\bibliography{ref}
\end{document}